\documentclass[11pt,reqno]{amsart}

\usepackage{setspace}
\usepackage{amsmath}
\usepackage{amssymb}
\usepackage{amsthm}
\usepackage{float}
\usepackage{graphicx}
\usepackage{subfigure}
\usepackage{amsfonts, amstext}
\usepackage{dsfont}
\usepackage{bm}
\usepackage{hyperref}
\usepackage[top=1.5in, bottom=1.5in, left=1.25in, right=1.25in]{geometry}

\numberwithin{equation}{section}

\newtheorem{theorem}{Theorem}[section]
\newtheorem{lemma}[theorem]{Lemma}

\newtheorem{corollary}[theorem]{Corollary}

\theoremstyle{remark}
\newtheorem{remark}{Remark}
\newtheorem{example}[theorem]{Example}
\newtheorem{definition}[theorem]{Definition}

{\bf}{\rm}

\newcommand{\PR}{\mathbb P}

\newcommand{\ER}{\mathbb E}

\newcommand{\norm}[1]{\left\lVert#1\right\rVert}

\DeclareMathOperator{\R}{\mathbb{R}}
\DeclareMathOperator{\N}{\mathbb{N}}
\DeclareMathOperator*{\argmin}{arg\,min}
\begin{document}


\title[Dimension Reduction in Vertex ERGM]{Dimension reduction in vertex-weighted exponential random graphs}

\author[R. DeMuse]{Ryan DeMuse}
\author[M. Yin]{Mei Yin}

\address{Department of Mathematics, University of Denver, Denver, CO 80208,
USA} \email{ryan.demuse@du.edu} \email{mei.yin@du.edu}


\begin{abstract}

We investigate the behavior of vertex-weighted exponential random graphs. We show that vertex-weighted exponential random graphs with edge weights induced by products of independent vertex weights are approximate mixtures of graphs whose vertex weight vector is a near fixed point of a certain vector equation. For graphs with Hamiltonians counting cliques, it is demonstrated that, under appropriate conditions, every solution to this equation is close to a block vector with a small number of communities. We prove that for the cases of positive weights and small weights in the Hamiltonian in particular, the vector equation has a unique solution. Lastly, the behavior of vertex-weighted exponential random graphs counting triangles is studied in detail and the solution to the vector equation is shown to approach the zero vector as the weight diverges to negative infinity for sufficiently large networks.

\vskip.1truein


\end{abstract}

\maketitle


\section{Introduction}

Complex networks have become omnipresent structures, especially with the popularity of technological and social networks, but also in economics, healthcare, and infrastructure \cite{F1} \cite{F2} \cite{J}. Their ubiquity has hastened the need to develop new models that aid in the study of these elaborate networks and techniques to study their local and global structures, as well as their development over time. One of the features of many modern networks that makes them difficult to study is the fact that one cannot assume, and is generally untrue, that edges between different vertices are independent of each other. The random graph model $G(n,p)$ is a well-studied probabilistic graph model where edges are placed between the $n$ vertices independently  with identical probability $p$ \cite{G}. This independence assumption is one of the characteristics of $G(n,p)$ that makes it so amenable to study. Unfortunately, this simultaneously makes $G(n,p)$ inadequate for modeling many modern networks. Think of the example of social networks, whose popularity has skyrocketed in recent years. Consider Persons $A$, $B$, and $C$. If $A$ is a friend of $B$ and $B$ is a friend of $C$, it is more likely for $A$ to be a friend of $C$ than it would be if there were no friendships between $A$ and $B$, and $B$ and $C$. This transitivity of friendship leads us to make the assumption that edges in social networks are unlikely to appear independently. Similarly, the mechanism that suggests new friends for users in social networks generically works upon the presumption that if Person $X$ is a friend to a large number of your friends, there is a greater chance that you would like to connect with Person $X$ than with some random user. 

The exponential random graph model (ERGM) is a popular model for random graphs, which assigns exponential families of probability distributions over spaces of graphs. Let $\mathcal{G}_{n}$ be the space of all finite simple graphs on $n$ vertices and $G \in \mathcal{G}_{n}$. The probability of $G$ in an exponential random graph model is given by
	\begin{equation} \label{uwm}
		\PR(G) = \exp\left(n^2 \left(f(G) - \psi_{n}^{f}\right)\right),
	\end{equation}
where $f : \mathcal{G}_{n} \to \R$ is some function on the graph space, and $\psi_{n}^{f}$ is a normalization constant given by
	\begin{equation}
		\psi_{n}^{f} = \frac{1}{n^2} \log \sum_{G \in \mathcal{G}_{n}} e^{n^2 f(G)}.
	\end{equation}
The function $f$ is referred to as the \textit{Hamiltonian} and typically consists of a weighted sum of subgraph counting functions. We remark that it is in general not necessary for $f$ to be a sum of subgraph counting functions. The Hamiltonian could be some other function on the graph space such as the degree distribution; however, in the present setting we will mostly concentrate on $f$ counting subgraphs. Some popular, well-studied Hamiltonians include the edge-two-star and edge-triangle models. For example, the Hamiltonian of the edge-triangle model is given by
	\begin{equation}
		f(G) = \alpha_1 t(K_2, G) + \alpha_2 t(K_3, G),
	\end{equation}
where $\alpha_1, \alpha_2 \in \R$ are weights, $K_k$ is the complete graph on $k$ vertices, and $t(K_k,G)$ is a homomorphism density function. For any finite simple graphs $H$ and $G$, define
	\begin{equation}
		t(H,G) = \frac{\mathsf{hom}(H, G)}{\left|V(G)\right|^{\left|V(H)\right|}},
	\end{equation}
where $\mathsf{hom}(H, G)$ is the number of edge-preserving homomorphisms of $H$ in $G$, and $V(H)$, $V(G)$ are the vertex sets of $H$ and $G$, respectively. This quantity can be interpreted as the probability that a random map from the set $V(G)^{V(H)}$ is edge-preserving.

Even though exponential random graph models on simple graphs already display a wide variety of behaviors and modeling capabilities, the assumption of non-weighted simple graphs still renders them somewhat insufficient when it comes to the modeling of many intricate complex networks \cite{CD} \cite{K}. Social networks exhibit clustering of vertices which corresponds to the formation of cliques as an important subgraph structure of a social network. Groups of friends or colleagues in a social network tend to form clusters, where the number of friendship relations in the cluster is close to the maximum possible. In the present work, we place weights on the vertices to accommodate this phenomenon. To start, we restrict our attention to vertex-weighted exponential random graph models with vertex weights having values $0$ or $1$. The case of vertex weights being $0$-$1$ random variables encourages the formation of cliques if the weight is one, and discourages their formation if the weight is zero. We assume that the vertices have weight $1$ with probability $p$ and $0$ with probability $1-p$. Thus we obtain a distribution on the space of $0$-$1$ vertex-weighted graphs such that any vertex-weighted graph $G$ has probability
	\begin{equation} \label{vwgd}
		\PR(G) = \exp\left(n \left(f(G) - \psi_{n}^{f}\right)\right) \PR_{n}(G),
	\end{equation}
where the normalization constant
	\begin{equation} \label{nc}
		\psi_{n}^{f} = \frac{1}{n} \log \ER_n \left( e^{n f(G)} \right),
	\end{equation}
and $\PR_{n}$ and $\ER_n$ are respectively the probability and expectation induced by the distribution of vertex weights.

There has been previous work done in the area of vertex-weighted ERGM including \cite{DEY}, where the efficiency of the Glauber dynamics on vertex-weighted ERGM was considered, expanding upon \cite{BBS}. This work builds upon the framework of \cite{EG2}; therein, Eldan and Gross showed that any exponential random graph model where the Hamiltonian counts subgraphs is close in $1$-norm distance to a mixture of stochastic block models with a small number of blocks (Theorem 14 in \cite{EG2}). Furthermore, the authors showed that if the weights are positive or sufficiently small (either positive or negative), there is a constant matrix $X_c$, such that every mixture element is close in $1$-norm distance to the constant matrix (Theorems 18 and 19 in \cite{EG2}). This implies that mixture elements behave like a $G(n,p)$ random graph since constant stochastic block matrices are equivalent to the random graph model $G(n,p)$.

We extend the results given in Theorems 14, 18, and 19 of \cite{EG2} to the vertex-weighted setting. The probability distribution defined in Equation \ref{vwgd} may be explicitly thought of as a distribution on the vector of vertex weights. That is, for $X \in \{0,1\}^{n}$, let
	\begin{equation} \label{vwcd}
		\PR(X)=\exp\left(f(X)- n\psi_{n}^{f}\right) \PR_{n}(X),
	\end{equation}
where for ease of computation we incorporate the number of vertices $n$ into the Hamiltonian and $\psi_{n}^{f}$ is similarly defined as in Equation \ref{nc}. In the context of $0$-$1$ vertex weights, the underlying distribution on the vertex weight configurations becomes
	\begin{equation}
		\PR_{n}(X) = p^{\norm{X}_1}(1-p)^{n - \norm{X}_1},
	\end{equation}
where $\norm{\cdot}_1$ is the $1$-norm of the vector $X$. The results of \cite{EG2} and, by extension, this paper rely on the set of near fixed-points of a certain function on the vertex weight configurations, specifically the set of $X$ such that
	\begin{equation} \label{nfp}
		\norm{X - \left(\mathbf{1}_n + \tanh(\nabla f(X))\right)/2}_1 \le cn^{7/8}
	\end{equation}
for some constant $c$, where $\tanh$ is applied element-wise, $\mathbf{1}_{n}$ is the $n$-dimensional all ones vector, and $\nabla f(X)$ is the discrete gradient of the subgraph counting function $f$ as defined in Equation \ref{grad}. Theorem \ref{fixedset} shows that for any Hamiltonian counting cliques, there exists a mixture measure such that the support of the measure is, asymptotic in $n$, given by the set of configurations satisfying Inequality \ref{nfp}. Theorem \ref{poscomm} further shows that any vertex weight configuration on $n$ vertices is close to a configuration with the number of vertex communities independent of $n$ in $1$-norm distance. Theorems \ref{fpsolution} and \ref{smallweights} prove the existence and uniqueness of a constant vector such that every mixture element is close to this constant vector when the weights are positive or sufficiently small in absolute value. In other words, we can partition the vertex set into communities where two vertices in the same community have the same vertex weight. This allows one to reduce the dimension by considering the fixed number of communities and the weight that vertices common to the same community share.

\section{Background}

Let $n > 0$ be the number of vertices. Let the set $\mathcal{C}_n = \{0,1\}^{n}$ represent the $n$-dimensional, discrete, Boolean hypercube and $\overline{\mathcal{C}_{n}} = [0,1]^{n}$ be the continuous hypercube. We also use the convention $[n] = \{1,2,\dots,n\}$. A vector $X \in \mathcal{C}_{n}$ represents a vertex-weight configuration for a graph on $n$ vertices, where $X_{i}$ is the weight at vertex $i$, and similarly for $X \in \overline{\mathcal{C}_{n}}$. Let $f: \mathcal{C}_{n} \to \R$ be a subgraph counting function of the form
	\begin{equation} \label{scf}
		f(X) = n \left( \sum_{q=1}^{l} \frac{\alpha_{q}}{n^{m_{q}}} \sum_{i_1 \neq \dots \neq i_{m_{q}}} X_{i_1} \cdots X_{i_{m_q}} \right) + \log\left( \frac{p}{1-p} \right) \norm{X}_1,
	\end{equation}
where $m_q \ge 2$ are distinct integers, $\alpha_q \in \R$, and $p \in (0,1)$. Note that the probability distribution induced by the vertex weights has been subsumed into the Hamiltonian. That is, we are considering a subgraph counting function that counts the number of certain cliques that appear in the graph with vertex weights given by $X \in \{0,1\}^{n}$. Often we will take $m_1 = 2$ counting the number of edges. Consider an edge $e = \{i,j\}$ present if the product $X_i X_j = 1$. This only occurs if both $X_i$ and $X_j$ are vertex weight $1$. As previously stated, we let $\norm{\cdot}_{1}$ be the $1$-norm of a vector in $\mathcal{C}_{n}$ and $\norm{\cdot}_{2}$ be the $2$-norm so that for $X \in \mathcal{C}_{n}$,
	\begin{equation*}
		\norm{X}_{1} = \sum_{i=1}^{n} X_i \hspace{.5cm}\text{and}\hspace{.5cm} \norm{X}_{2} = \sqrt{\sum_{i=1}^{n} X_i^2}.
	\end{equation*}

\subsection{Graph Limits and the Variational Problem}

In their seminal paper \cite{ChDi}, Chatterjee and Diaconis utilized the framework of graph limit theory to show that certain exponential random graph models behave like Erd\H{o}s-R\'enyi random graphs. We present some of the basics of graph limit theory here so that we may state a few of their results. Let $\mathcal{W}$ be the space of all symmetric, measurable functions $h : [0,1]^{2} \to [0,1]$, called the \textit{graphon} or \textit{graph limit space}. From any simple graph $G$ on $n$ vertices, one can construct an associated graphon, $h^{G}$, such that
	\begin{equation*}
		h^{G}(x,y) = \left\{ \begin{array}{ll} 1 & \text{if }\left\{\left\lceil nx \right\rceil, \left\lceil ny \right\rceil\right\} \text{is an edge in } G, \\ 0 & \text{otherwise.} \\  \end{array} \right.
	\end{equation*}
Define the \textit{cut norm} on the graphon space $\mathcal{W}$ as
	\begin{equation*}
		\norm{h}_{\square} = \sup_{A,B \subseteq [0,1]} \left| \int_{A \times B} h(x,y) \,dx \,dy \right|.
	\end{equation*}
This gives a pseudometric such that for $h,g \in \mathcal{W}$,
	\begin{equation*}
		d_{\square}(h,g) = \norm{h-g}_{\square}.
	\end{equation*}
The induced metric is given by identifying equivalent graphons $h,g$ such that $h \sim g$ if they agree on a set of full measure or there exists a measure preserving bijection $\sigma: [0,1] \to [0,1]$ such that $h(x,y) = g(\sigma x, \sigma y) = g_{\sigma}(x,y)$. Denote by $\tilde{h}$ the closure in $(\mathcal{W}, d_{\square})$ of the orbit $\{h_\sigma\}$. A quotient space $\widetilde{\mathcal{W}}$ is obtained, referred to as the \textit{reduced} or \textit{unlabeled graphon space}, and $d_{\square}$ becomes a metric on this space, $\delta_{\square}$, by defining
	\begin{equation*}
		\delta_{\square}\left(\tilde{h}, \tilde{g}\right) = \inf_{\sigma} d_{\square}\left( h, g_{\sigma} \right).
	\end{equation*}
It follows that $(\widetilde{\mathcal{W}}, \delta_{\square})$ is a compact metric space and the homomorphism density functions are continuous with respect to the cut distance \cite{LL}.

Let $f$ be a Hamiltonian on the reduced graphon space and consider $\psi_{\infty}(f) := \lim_{n \to \infty} \psi_{n}(f)$, known as the \textit{limiting normalization constant}. Define $I: [0,1] \to \R$ as
	\begin{equation*}
		I(u) = \frac{1}{2} u \log u + \frac{1}{2} (1-u) \log(1-u),
	\end{equation*}
and extend $I$ to $\widetilde{\mathcal{W}}$ as
	\begin{equation*}
		I(\tilde{h}) = \int_{0}^{1} \int_{0}^{1} I(h(x,y)) \,dx \,dy,
	\end{equation*}
where $h$ is any representative element of the equivalence class $\tilde{h}$.

\begin{theorem}[Theorem 3.1 in \cite{ChDi}] \label{ChDi}
	If $f$ is a bounded continuous function on the reduced graphon space, then
		\begin{equation} \label{var}
			\psi_{\infty}(f) = \sup_{\tilde{g} \in \widetilde{\mathcal{W}}} \left\{ f(\tilde{g}) - I(\tilde{g}) \right\}.
		\end{equation}
\end{theorem}
Using this formulation of the limiting normalization constant, Chatterjee and Diaconis showed that a graph drawn from the model defined in Equation \ref{uwm} lies close in cut distance to the set of maximizers of the limiting normalization constant.

\begin{theorem}[Theorem 3.2 in \cite{ChDi}] \label{convprob}
	Let $G$ be drawn from $\PR$ defined in Equation \ref{uwm} and $\widetilde{F}^{\ast}$ be the set of maximizers of Equation \ref{var}. Then for all $\varepsilon > 0$ there exists a $C > 0$ and a $\delta > 0$ such that for all $n \in \N$,
		\begin{equation*}
			\PR \left( \delta_{\square} \left( \tilde{g}^{G}, \widetilde{F}^{\ast} \right) > \varepsilon  \right) \le Ce^{-n^2 \delta}.
		\end{equation*}
\end{theorem}
Specializing to subgraph counting function $f$ of the form
\begin{equation}
		f(G) = \sum_{q=1}^{l} \alpha_q t(H_q,G),
	\end{equation}
where $\alpha_q \in \R$ and $H_{q}$ are finite, simple graphs, Chatterjee and Diaconis further derived the following.

\begin{theorem}[Theorem 4.2 in \cite{ChDi}]
	If $H_1 = K_2$ and $\alpha_2, \dots, \alpha_m$ are nonnegative, then the set of maximizers $\widetilde{F}^{\ast}$ may be identified as a finite subset of $[0,1]$ and
		\begin{equation}
			\min_{u \in \widetilde{F}^{\ast}} \delta_{\square} (\tilde{g}^{G}, u) \to 0
		\end{equation}
	in probability as $n \to \infty$. That is, the sampled graph $G$ behaves like an Erd\H{o}s-R\'enyi random graph $G(n,u)$ where $u$ is picked randomly from some distribution on $\widetilde{F}^{\ast}$.
\end{theorem}

These results for standard exponential random graphs were extended to the realm of edge-weighted exponential random graphs by Yin in \cite{Yin}. However, by utilizing graphons and the cut distance, this approach is inherently limited to the dense regime. In this paper, we will prove corresponding results in both the dense and the sparse regime. Morever, the metric used in the current investigation is the $1$-norm distance, which induces a finer topology than the cut distance \cite{LL}.

\subsection{Definitions}

Several results in this paper are first proven for a single subgraph counting function, then generalized to sums using triangle inequalities and subadditivity. For a Hamiltonian counting a single type of subgraph, we have
	\begin{equation}
\label{single}
		f(X) = \frac{\alpha}{n^{m-1}} \left( \sum_{i_1 \neq \dots \neq i_{m}} X_{i_1} \cdots X_{i_{m}} \right) + \log\left( \frac{p}{1-p} \right) \norm{X}_1
	\end{equation}
for some positive integer $m \ge 2$, $\alpha \in \R$, and $p \in (0,1)$. A single subgraph counting function in this work therefore counts the number of $m$-cliques in the graph obtained from the vertex configuration $X$. For a subgraph counting function $f$ and $j \in [n]$, define the \textit{discrete derivative} of $f$ at vertex $j$ as
	\begin{equation*} 
		\partial_{j} f(X) = \frac{1}{2} \left( f\left(X_1,\dots,X_{j-1},1,X_{j+1},\dots,X_{n}\right) - f\left(X_1,\dots,X_{j-1},0,X_{j+1},\dots,X_{n}\right) \right).
	\end{equation*}
This further gives a definition of the \textit{discrete gradient} of the function $f$, $\nabla f(X)$, as the vector of discrete derivatives where
	\begin{equation} \label{grad}
		\nabla f(X) = \left( \partial_1 f(X),\dots, \partial_n f(X) \right),
	\end{equation}
and the Lipschitz constant of $f$ is given by
	\begin{equation}
		\mathsf{Lip}(f) = \max_{i \in [n], X \in \mathcal{C}_{n}} \left|\partial_{i} f(X) \right|.
	\end{equation}
It will be convenient to define $C_j^{m}$ for $j \in [n]$ and $m \ge 2$ as
	\begin{align}
		C_{j}^{m}(X) 
			&= \sum_{j \neq i_1 \neq \dots \neq i_{m-1}} X_{i_1} \cdots X_{i_{m-1}},
	\end{align}
with the dependence on $X$ suppressed when it is clear from the context. For $m = 1$ let $C_{j}^{1} = 1$, which corresponds to the sum of the empty product. With this notation and $f$ a single subgraph counting function, for $j \in [n]$, we obtain a more compact form of the discrete derivative of $f$ at coordinate $j$ as
 	\begin{equation} \label{partial}
 		\partial_{j} f(X) = \frac{m\alpha}{2 n^{m-1}} C_{j}^{m} + \frac{1}{2} \log\left( \frac{p}{1-p} \right).
 	\end{equation}
Let $\mathbf{C}^{m} \in \R^{n}$ be the vector $\mathbf{C}^{m} = (C_1^{m}, C_2^{m}, \dots, C_{n}^{m})$ for $m \in \N$. Since a generic subgraph counting function is a weighted sum of single subgraph counting functions, the discrete gradient of a generic subgraph counting function may then be written as a weighted sum of vectors $\mathbf{C}^{m}$. The following Lemma \ref{innerprodrep} shows that the vector $\textbf{C}^{m}$ can be defined inductively as a matrix product.

\begin{lemma} \label{innerprodrep}
	Let $J$ be the $n \times n$ matrix consisting of all ones and $I$ the $n \times n$ identity matrix. If $\mathbf{C}^{m}$ is defined as above, then
		\begin{equation}
			\mathbf{C}^{m} = \mathbf{C}^{m-1} \mathbf{X} (J - (m-1)I),
		\end{equation}
	where $\mathbf{X}$ is the matrix with the vector $X$ on the main diagonal and zeros off diagonal.
\end{lemma}
	\begin{proof}
		Firstly, note that we can write $C_{j}^{m}$ as
			\begin{equation} \label{id0}
				C_{j}^{m} = \left(\sum_{i=1}^{n} X_{i} C_{i}^{m-1}\right) - (m-1) X_j C_{j}^{m-1}.
			\end{equation}
		This can be argued as follows. We will start with the right hand side, which can be expanded as
			\begin{equation} \label{id1}
				\left(\sum_{i=1}^{n} X_{i} \sum_{i \neq i_1 \neq \cdots \neq i_{m-2}} X_{i_1} \cdots X_{i_{m-2}}\right) - (m-1) X_j \sum_{j \neq i_1 \neq \cdots \neq i_{m-2}} X_{i_1} \cdots X_{i_{m-2}}.
			\end{equation}
		The expression in parentheses in Expression \ref{id1} is equivalent to
			\begin{equation}
				\sum_{i_1 \neq \cdots \neq i_{m-1}} X_{i_1} \cdots X_{i_{m-1}},
			\end{equation}
		while the term on the right represents all terms involving the variable $X_j$. Note that there are $m-1$ choices to place the variable $X_j$ in the product. Once $X_j$ is factored out, we simply multiply by all possible nomials of distinct terms not involving $X_j$. Equation \ref{id0} thus holds. Using this equation, we find
			\begin{align}
				\mathbf{C}^{m-1} \mathbf{X} (J - (m-1)I) &= \left(X_1 C_{1}^{m-1}, \dots, X_n C_{n}^{m-1}\right)(J-(m-1)I) \notag \\
					&= \left( \sum_{j=1}^{n} X_{j} C_{j}^{m-1} \right) \mathbf{1}_{n} - (m-1) \left( X_1 C_{1}^{m-1}, \dots, X_{n} C_{n}^{m-1} \right) \notag \\
					&= \mathbf{C}^{m},
			\end{align}
		and the desired expression readily follows.
	\end{proof}

Note that we can also write $\mathbf{C}^{m}$ as an elementary matrix product of the form
	\begin{equation}
		\mathbf{C}^{m} = \mathbf{1}_{n} \prod_{i=1}^{m-1} \mathbf{X}(J-iI),
	\end{equation}
with $J$, $I$, and $\mathbf{X}$ as defined in Lemma \ref{innerprodrep}. Theorem \ref{ergmsbm}, which appears as a main structural theorem in \cite{EG}, will constitute a backbone of one of the main results of this paper, showing that vertex-weighted exponential random graphs are close to mixtures of independent graph models. The formulation of this theorem requires the concept of the \textit{Gaussian-width} of a subset of $\R^n$.

\begin{definition}\label{gc}
	The \textbf{Gaussian-width} of a set $K \subseteq \R^{n}$ is defined as
		\begin{equation}
			\mathsf{GW}(K) = \ER\left( \sup_{X \in K} \left\langle X, \Gamma \right\rangle \right),
		\end{equation}
	where $\Gamma \sim N\left(0, I\right)$ is a standard normal Gaussian vector in $\R^{n}$. For a function $f : \mathcal{C}_{n} \to \R$, the \textit{gradient complexity} of $f$ is defined as
		\begin{equation}
			\mathcal{D}(f) = \mathsf{GW}\left( \left\{ \nabla f(X) : X \in \mathcal{C}_{n} \right\} \cup \{0\} \right).
		\end{equation}
\end{definition}

\begin{definition}
	For $\mathbf{p} \in [0,1]^{n}$, denote by $G(n,\mathbf{p})$ the random vertex-weighted graph such that vertex $i$ has weight $\mathbf{p}_{i}$ and edge $i \sim j$ appears with probability $\mathbf{p}_i \mathbf{p}_j$. Let $\rho$ be a probability measure on $[0,1]^{n}$. Define $G(n,\rho)$ as
		\begin{equation}
			\PR\left( G(n,\rho) = G \right) = \int \PR\left( G(n,\mathbf{p}) = G \right) \,d\rho(\mathbf{p}),
		\end{equation}
	and we say $G(n,\rho)$ is a \textbf{$\rho$-\textit{mixture}}.
\end{definition}

\begin{definition}
	Let $\delta > 0$ and let $\rho$ be a probability measure on $[0,1]^{n}$. A random vertex-weighted graph $G$ is called a \textbf{$(\rho,\delta)$-mixture} if there exists a coupling between $G(n,\rho)$ and $G$ such that
		\begin{equation}
			\ER\norm{G(n,\rho) - G}_1 \le \delta n.
		\end{equation}
\end{definition}

\begin{definition}
	A \textbf{block vector with $k$ communities} is a vector $v \in \R^{n}$ such that there exists a set $U \subseteq \R$ with $\left|U\right| = k$ and $v_{i} \in U$ for all $1 \le i \le n$.
\end{definition}

\section{Setup}

We wish to use Theorem 9 of \cite{EG} to show that vertex-weighted exponential random graphs behave like $G(n,\rho)$ mixtures for some measure $\rho$. The theorem is stated as follows.

\begin{theorem}[Theorem 9 in \cite{EG}] \label{ergmsbm}
	Let $n > 0$, $f: \mathcal{C}_{n} \to \R$ a Hamiltonian, and $X_{n}^{f}$ a random vector given by
		\begin{equation}
			\PR \left( X_{n}^{f} = X \right) = \exp(f(X))/Z,
		\end{equation}
	where $Z$ is a normalizing constant. Denote
		\begin{align}
			&D = \mathcal{D}(f), \\
			&L_1 = \max \left\{ 1, \mathsf{Lip}(f) \right\}, \\
			&L_2 = \max \left\{ 1, \max_{X \neq Y \in \mathcal{C}_{n}} \frac{\norm{\nabla f(X) - \nabla f(Y)}_1}{\norm{X-Y}_1} \right\}.
		\end{align}
	Furthermore, define $\mathcal{X}_{f}$ to be the set
		\begin{equation}
			\mathcal{X}_f = \left\{ X \in \overline{\mathcal{C}_n} : \norm{X - \frac{\mathbf{1}_{n} + \tanh(\nabla f(X))}{2}}_1 \le 5000 L_1 L_{2}^{3/4} D^{1/4} n^{3/4} \right\},
		\end{equation}
	where $\mathbf{1}_{n}$ is the $n$-dimensional all ones vector, $\nabla f(X)$ is extrapolated to $\overline{\mathcal{C}_{n}}$ by Equation \ref{partial} and with $\tanh$ applied entry-wise. Then $X_{n}^{f}$ is a $\left( \rho, 80 \frac{D^{1/4}}{n^{1/4}} \right)$-mixture such that
		\begin{equation}
			\rho\left( \mathcal{X}_{f} \right) \ge 1 - 80 \frac{D^{1/4}}{n^{1/4}}.
		\end{equation}
	In particular, if $D = o(n)$, then $X_{n}^{f}$ is a $(\rho, o(1))$-mixture with $\rho\left( \mathcal{X}_{f} \right) = 1 - o(1)$.
\end{theorem}

In order to use Theorem \ref{ergmsbm}, we must establish bounds on $D, L_1$, and $L_2$. Throughout this Section, we assume that $f$ is a subgraph counting function of the form in Equation \ref{scf}. Consider first the determination of $D$. Using Lemma \ref{innerprodrep} and the definition of Gaussian-width we derive a bound on $\mathcal{D}(f)$ for any $f$ of the form \ref{scf}.

\begin{lemma} \label{gradientcomp}
	Let $f$ be a subgraph counting function of the form in Equation \ref{scf}. Then
		\begin{equation}
			\mathcal{D}(f) \le \frac{\sqrt{n}}{2} \left(\sum_{q=1}^{l} m_q\left|\alpha_{q}\right| \sqrt{1 + \frac{(2-m_{q})^2 - 1}{n}} \right).
		\end{equation}
\end{lemma}

	\begin{proof}
		For a single subgraph counting function, by Lemma \ref{innerprodrep}, the gradient of $f$ has the form
			\begin{equation}
				\nabla f(X) = \frac{1}{2} \frac{m\alpha}{n^{m-1}} \mathbf{C}^{m-1} \mathbf{X} (J - (m-1)I) + \frac{1}{2} \log\left(\frac{p}{1-p}\right) \mathbf{1}_n.
			\end{equation}
		Recall Definition \ref{gc} for the gradient complexity $\mathcal{D}(f)$. Taking the expectation over $\Gamma$ yields
			\begin{align}
				\mathcal{D}(f)
					&\le \frac{1}{2} \ER \sup_{X \in \mathcal{C}_n} \left| \left\langle \frac{m\alpha}{n^{m-1}} \mathbf{C}^{m-1} \mathbf{X} (J - (m-1)I), \Gamma \right\rangle \right| \notag \\
						&\hspace{1cm}+ \frac{1}{2} \ER \left\langle \log\left(\frac{p}{1-p}\right) \mathbf{1}_n, \Gamma \right\rangle \notag  \\
					&\le  \frac{m\left|\alpha\right|}{2} \ER \sup_{X \in \mathcal{C}_n} \left| \left\langle n^{-(m-2)} \mathbf{C}^{m-1} \mathbf{X}, n^{-1}(J-(m-1)I)\Gamma \right\rangle \right| \notag \\
					&\le \frac{m\left|\alpha\right|}{2} \sqrt{ \sup_{X \in \mathcal{C}_{n}} \norm{n^{-(m-2)} \mathbf{C}^{m-1} \mathbf{X}}_{2}^{2} \ER \norm{n^{-1} (J-(m-1)I)\Gamma}_{2}^{2}} \notag \\
					&\le \frac{m\left|\alpha\right|}{2} \sqrt{ n \, \mathsf{tr}\left( (n^{-1} (J-(m-1)I))^{2} \right)} \notag \\
					&= \frac{m\left|\alpha\right|}{2} \sqrt{n \left( 1 + \frac{(2-m)^2 - 1}{n} \right)}.
			\end{align}
We now turn to generic subgraph counting functions with multiple components. Take $f$ a subgraph counting function of the form \ref{scf},
		where $m_1, \dots, m_{l}$ are distinct positive integers with $m_i \ge 2$ and $\alpha_i \in \R$. Using the subadditivity of the gradient complexity, the conclusion readily follows. The bound on $\mathcal{D}(f)$ may be relaxed further as
			\begin{equation}
				\mathcal{D}(f) \le \frac{\sqrt{n}}{2} \sum_{q=1}^{l} \left|\alpha_{q}\right| m_q (m_q - 1).
			\end{equation}
\end{proof}

Next we examine $\mathsf{Lip}(f)$. From Equation \ref{partial}, we know that for $j \in [n]$,
	\begin{equation*}
		\partial_{j} f(X) = \frac{m\alpha}{2 n^{m-1}} C_{j}^{m} + \frac{1}{2} \log\left( \frac{p}{1-p} \right) \le \frac{m\alpha (n-1)_{m-1}}{2 n^{m-1}} + \frac{1}{2} \left| \log\left( \frac{p}{1-p} \right) \right|,
	\end{equation*}
where $(n)_{m}$ is the falling factorial,
	\begin{equation*}
		(n)_{m} = n(n-1)(n-2)\cdots(n-m+1).
	\end{equation*}
This upper bound comes from assigning each vertex weight to $1$, thus obtaining the number of permutations of $m-1$ distinct variables from a total of $n-1$ variables. Thus for $f$ a general subgraph counting function of the form \ref{scf}, $\mathsf{Lip}(f)$ may be bounded as
	\begin{align}
		\mathsf{Lip}(f) \le \frac{1}{2} \left|\log\left(\frac{p}{1-p}\right)\right| + \frac{1}{2} \sum_{q=1}^{l} m_q\left|\alpha_{q}\right|.
	\end{align}

With the necessary bounds on $D$ and $L_1$ established for Theorem \ref{ergmsbm}, we find an upper bound on $L_2$, i.e. a bound on the quantity
	\begin{equation}
		\max_{X \neq Y \in \mathcal{C}_{n}} \frac{\norm{\nabla f(X) - \nabla f(Y)}_1}{\norm{X - Y}_1}.
	\end{equation}
To do so, we first bound the distance between two vertex weight configurations that differ at a single vertex in Lemma \ref{partialbound}, then use the triangle inequality to bound the $1$-norm distance between configurations that differ at any number of vertices.

\begin{lemma} \label{partialbound}
	Let $X,Y \in \mathcal{C}_{n}$ be two vectors that differ at a single coordinate $i$. Then for any $j \in [n]$,
		\begin{equation}
			\left| \partial_{j} f(X) - \partial_{j} f(Y) \right| \le \frac{1}{2} \sum_{q=1}^{l} \frac{\left|\alpha_{q}\right| m_q(m_q - 1)}{n}.
		\end{equation}
\end{lemma}

	\begin{proof}
		Suppose that $X_i = 1$ and $Y_i = 0$. Note that if $i = j$, then
			\begin{equation}
				\left| \partial_{j} f(X) - \partial_{j} f(Y) \right| = 0.
			\end{equation}
		Assume that $i \neq j$, then from Equation \ref{partial},
			\begin{align}
				\left| \partial_{j} f(X) - \partial_{j} f(Y) \right| &\le \frac{1}{2} \sum_{q=1}^{l} \frac{m_q\left|\alpha_q\right|}{n^{m_q - 1}} \left| \sum_{j \neq i_1 \neq \dots \neq i_{m_{q}-1}} \left(X_{i_1}\cdots X_{i_{m_q - 1}} - Y_{i_1}\cdots Y_{i_{m_q - 1}}\right) \right| \notag\\
					&= \frac{1}{2} \sum_{q=1}^{l} \frac{m_q\left|\alpha_q\right|}{n^{m_q - 1}} \left| (m_q - 1)! \sum_{\substack{P \subseteq [n] \setminus \{i,j\} \\ \left|P\right| = m_{q} - 2}} \prod_{p \in P} X_{p} \right| \notag\\
					&\le \frac{1}{2} \sum_{q=1}^{l} \frac{\left|\alpha_q\right|m_q! }{n^{m_q - 1}} \frac{(n-2)(n-3) \cdots (n-m_{q}+1)}{(m_q-2)!} \notag\\
					&\le \frac{1}{2} \sum_{q=1}^{l} \frac{\left|\alpha_q\right|m_q(m_q - 1) }{n}.
			\end{align}
	\end{proof}

Using Lemma \ref{partialbound}, we may bound the $1$-norm distance between the gradient vectors of any two vertex-weight configurations.

\begin{lemma} \label{normdistpreserved}
	Let $f$ be a subgraph counting function of the form \ref{scf} and $X,Y \in \overline{\mathcal{C}_{n}}$. Then
		\begin{equation}
			\norm{\nabla f(X) - \nabla f(Y)}_{1} \le C \norm{X - Y}_{1},
		\end{equation}
	where
		\begin{equation}
			C = \frac{1}{2} \sum_{q=1}^{l} \left|\alpha_q\right| m_q (m_q - 1).
		\end{equation}
\end{lemma}

	\begin{proof}
		Suppose that $X$ and $Y$ differ in a single coordinate $i$. Holding all other coordinates fixed, $\nabla f$ is linear as a
function of the $i$th coordinate.  Then using Lemma \ref{partialbound},
			\begin{align}
				\norm{\nabla f(X) - \nabla f(Y)}_{1} &= \sum_{j=1}^{n} \left| \partial_{j} f(X) - \partial_{j} f(Y) \right| \notag\\
					&\le \sum_{q=1}^{l} \sum_{j=1}^{n} \frac{\left|\alpha_q\right|m_q(m_q - 1)}{2n} \left|X_i - Y_i\right| \notag\\
					&= \sum_{q=1}^{l} \frac{\left|\alpha_q\right|m_q(m_q - 1)}{2} \left|X_i - Y_i\right|.
			\end{align}
		By the triangle inequality, we extend this to vectors that differ in more than one coordinate. The desired result is obtained.
	\end{proof}

\section{Main Results}

With the bounds in Theorem \ref{ergmsbm} established in Lemmas \ref{gradientcomp}, \ref{partialbound}, and \ref{normdistpreserved}, we are ready to formulate our first main result Theorem \ref{fixedset}. Fix a subgraph counting function $f$ of the form in Equation \ref{scf}. Let $C_{\alpha}$ be the constant
	\begin{equation} \label{calpha}
		C_{\alpha} = \max \left\{ 2, \frac{1}{2} \left| \log\left(\frac{p}{1-p}\right) \right| + \frac{1}{2} \sum_{q=1}^{l} \left|\alpha_{q}\right| m_q (m_q - 1) \right\}.
	\end{equation}
Furthermore, let $\mathcal{X}_{f} \subseteq [0,1]^{n}$ be defined such that
	\begin{equation}
		\mathcal{X}_{f} = \left\{ X \in [0,1]^{n} \, : \, \norm{X - \frac{\mathbf{1}_{n} + \tanh(\nabla f(X))}{2}}_{1} \le 5000 C_{\alpha}^{2} n^{7/8}  \right\}.
	\end{equation}
Theorem \ref{ergmsbm} gives that vertex-weighted exponential random graphs behave like mixture models. Additionally, the majority of the weight in this mixture model lies on configurations that are almost fixed-points to the vector equation
	\begin{equation}\label{vector}
		X = \frac{\mathbf{1}_{n} + \tanh(\nabla f(X))}{2}.
	\end{equation}

\begin{theorem} \label{fixedset}
	Let $f$ be a subgraph counting function where the subgraphs are counted as weighted, normalized products of vertex weights as in Equation \ref{scf}. Then there exists a measure $\rho$ on $[0,1]^{n}$ dependent on $n$ and $f$ such that $X_{n}^{f}$ is a $\left( \rho, 80 C_{\alpha}^{1/4} n^{-1/8} \right)$-mixture such that
		\begin{equation}
			\rho(\mathcal{X}_{f}) \ge 1 - 80 C_{\alpha}^{1/4} n^{-1/8}.
		\end{equation}
\end{theorem}

We will use this preceding theorem to show in Theorem \ref{poscomm} that if $X \in \mathcal{X}_{f}$, then there is a block vector $X^{\ast}$ with a small number of communities such that $X$ is close in $1$-norm distance to $X^{\ast}$. Toward this end, we will require some statements concerning the preservation of length under orthogonal projections.

\begin{lemma}[Theorem 2.1 in \cite{DG}] \label{normspreserved}
	Let $0 < \delta < 1$, let $d,k > 0$ be positive integers, let $\pi: \R^{d} \to \R^{k}$ be an orthogonal projection into a uniformly random $k$ dimensional subspace, and let $g : \R^{d} \to \R^{k}$ be defined as $ g(v) = \sqrt{\frac{d}{k}} \pi(v)$. Then for any vector $v \in \R^{d}$,
		\begin{equation}
			\mathbb{P} \left( (1-\delta)\norm{v}_{2}^{2} \le \norm{g(v)}_{2}^{2} \le (1+\delta) \norm{v}_{2}^{2} \right) \ge 1 - 2e^{-k\left( \delta^2 / 2 - \delta^{3} / 3 \right)/2}.
		\end{equation}
\end{lemma}

\begin{lemma}[Lemma 30 in \cite{EG2}] \label{dotdistpreserved}
	Let $0 < \delta < 1$, $d$ and $k$ be positive integers, and $g : \R^{d} \to \R^{k}$ be a linear transformation. If $u,v \in \R^{d}$ are two vectors of norm smaller than $1$ such that
		\begin{equation}
			(1-\delta)\norm{u \pm v}_{2}^{2} \le \norm{g(u \pm v)}_{2}^{2} \le (1+\delta)\norm{u \pm v}_{2}^{2},
		\end{equation}
	then
		\begin{equation}
			\left| \left\langle g(u), g(v) \right\rangle - \left\langle u, v \right\rangle \right| \le 2\delta.
		\end{equation}
\end{lemma}

\begin{theorem} \label{poscomm}
	Let $0 < \delta < 1$ and let $f$ be a subgraph counting function of the form \ref{scf}. Then there exists a constant $C_{\delta} > 0$, independent of $n$, such that for any $X \in \mathcal{X}_{f}$, there exists a block vector $X^{\ast}$ with no more than $C_{\delta}$ communities such that
		\begin{equation}
			\norm{X - X^{\ast}}_{1} \le \delta n + 5000 C_{\alpha}^{2} n^{7/8}.
		\end{equation}
\end{theorem}

	\begin{proof}
		Let $f$ be a subgraph counting function of the form defined in Equation \ref{scf} and let $X \in \mathcal{X}_{f}$. That is,
			\begin{equation*}
				f(X) = \log\left(\frac{p}{1-p}\right) \norm{X}_{1} +  \sum_{q = 1}^{l} \frac{\alpha_{q}}{n^{m_{q}-1}} \sum_{i_1 \neq \cdots \neq i_{m_{q}}} X_{i_1} \cdots X_{i_{m_{q}}}.
			\end{equation*}
		As in Lemma \ref{gradientcomp}, we may write
			\begin{equation*}
				\nabla f(X) = \frac{1}{2} \log\left(\frac{p}{1-p}\right) \mathbf{1}_{n} + \frac{1}{2} \sum_{q = 1}^{l} \frac{m_q \alpha_q}{n^{m_{q}-1}} \mathbf{C}^{m_{q}-1}\mathbf{X}(J - (m_{q}-1)I),
			\end{equation*}
		and the components of the gradient may be written as inner products such that for $j\in [n]$,
			\begin{equation*}
				\partial_{j} f(X) = \frac{1}{2} \log\left(\frac{p}{1-p}\right) + \frac{1}{2} \sum_{q = 1}^{l} m_q\alpha_{q} \sqrt{1 - \frac{1}{n} + \frac{(2-m_q)^2}{n}} \left\langle u^{q}, v_{j}^{q} \right\rangle,
			\end{equation*}
		where
			\begin{equation*}
				u^{q} = \frac{1}{n^{m_q - 2}\sqrt{n}} \mathbf{C}^{m_{q}-1} \mathbf{X}
			\end{equation*}
		and
			\begin{equation*}
				v_{j}^{q} = \frac{1}{\sqrt{n - 1 + (2-m_q)^2}} \left(1,1,\dots,2-m_{q},\dots,1,1\right),
			\end{equation*}
		with the entry $2-m_{q}$ in the $j$-th position. Note that $v_{j}^{q}$ is a multiple of the $j$-th column of the matrix $J-(m_{q}-1)I$. For each $q = 1, \dots, l$ and $j = 1, \dots, n$, the construction yields $\norm{u^{q}}_{2}^{2}, \norm{v_{j}^{q}}_{2}^{2} \le 1$.

		Let $0 < \delta < 1$, $k = \left\lceil 2\log(1/\delta) \left(\delta^2 / 2 - \delta^3 / 3\right)^{-1}\right\rceil$, $U \subseteq \R^{n}$ a uniformly random subspace of dimension $k$, and denote by $\pi : \R^{n} \to U$ an orthogonal projection from $\R^{n}$ into $U$. Take $g : \R^{n} \to U$ as $  g(v) = \sqrt{\frac{n}{k}} \, \pi(v)$. For $q = 1, \dots, l$ and $j = 1, \dots, n$, define
			\begin{equation*}
				\mathcal{B}_{j}^{q} = \left\{ (1-\delta) \norm{x}_{2}^{2} \le \norm{g(x)}_{2}^{2} \le (1+\delta)\norm{x}_{2}^{2} : x \in \left\{ u^{q}, v_{j}^{q}, u^{q} \pm v_{j}^{q} \right\}\right\}.
			\end{equation*}
		By Lemma \ref{normspreserved},
			\begin{equation*}
				\PR\left( \mathcal{B}_{j}^{q} \right) \ge 1 - 8 e^{-k\left(\delta^2 / 2 - \delta^3 / 3\right)/2}.
			\end{equation*}
		Under this event, both $g(u^{q})$ and $g(v_{j}^{q})$ lie in a ball of radius $2$ around the origin, since $u^{q}$ and $v_{j}^{q}$ lie in the unit ball and $\delta < 1$. It is also true by Lemma \ref{dotdistpreserved} that
			\begin{equation*}
				\left| \left\langle g(u^{q}), g(v_{j}^{q}) \right\rangle - \left\langle u^{q}, v_{j}^{q} \right\rangle \right| \le 2\delta.
			\end{equation*}
	Let $T$ be a $\delta$-net of the ball of radius $2$ about the origin in $k$ dimensions. By Lemma 2.6 of \cite{MS}, there exists such a net of size no larger than $(1+4/\delta)^{k+1}$. For each $q = 1, \dots, l$ and $j = 1, \dots, n$, define
			\begin{equation*}
				\overline{u^{q}} = \argmin_{w \in T} \norm{g(u^{q}) - w}_{2},
			\end{equation*}
			\begin{equation*}
				\overline{v_{j}^{q}} = \argmin_{w \in T} \norm{g(v_{j}^{q}) - w}_{2}.
			\end{equation*}
		Consider $\Delta \overline{u^{q}} = \overline{u^{q}} - g(u^{q})$ and $\Delta \overline{v_{j}^{q}} = \overline{v_{j}^{q}} - g(v_{j}^{q})$. Since under $\mathcal{B}_{j}^{q}$, $g(u^q)$ and $g(v_{j}^{q})$ are in the ball of radius $2$, both $\norm{\Delta \overline{u^{q}}}_{2}, \norm{\Delta \overline{v_{j}^{q}}}_{2} \le \delta$. This gives
			\begin{align*}
				\left| \left\langle \overline{u^{q}}, \overline{v_{j}^{q}} \right\rangle - \left\langle g(u^{q}), g(v_{j}^{q}) \right\rangle \right| &= \left| \left\langle \Delta \overline{u^{q}} + g(u^{q}), \overline{v_{j}^{q}} \right\rangle - \left\langle g(u^{q}), g(v_{j}^{q}) \right\rangle \right| \notag\\
					&= \left| \left\langle \Delta \overline{u^{q}}, \overline{v_{j}^{q}} \right\rangle + \left\langle g(u^{q}), \overline{v_{j}^{q}} - g(v_{j}^{q}) \right\rangle \right| \notag\\
					&= \left| \left\langle \Delta \overline{u^{q}}, \overline{v_{j}^{q}} \right\rangle + \left\langle g(u^{q}), \Delta \overline{v_{j}^{q}} \right\rangle \right| \notag\\
					&\le 4\delta.
			\end{align*}
		By the triangle inequality,
			\begin{align*}
				\left| \left\langle \overline{u^{q}}, \overline{v_{j}^{q}} \right\rangle - \left\langle u^{q}, v_{j}^{q} \right\rangle \right| &\le \left| \left\langle \overline{u^{q}}, \overline{v_{j}^{q}} \right\rangle - \left\langle g(u^{q}), g(v_{j}^{q}) \right\rangle \right| \notag\\
					&\hspace{1cm}+ \left| \left\langle g(u^{q}), g(v_{j}^{q}) \right\rangle - \left\langle u^{q}, v_{j}^{q} \right\rangle \right| \notag\\
					&\le 6\delta.
			\end{align*}

		For each $q = 1, \dots, l$ and $j = 1, \dots, n$, define $\overline{X_{j}^{q}} = \left\langle \overline{u^{q}}, \overline{v_{j}^{q}} \right\rangle$ and $X_{j}^{q} = \left\langle u^{q}, v_{j}^{q} \right\rangle$. Then $\overline{X^{q}}$ is a block vector with no more than $(1+4/\delta)^{2(k+1)}$ communities, and
			\begin{equation}
				\ER \norm{X^{q} - \overline{X^q}}_{1} = \ER \sum_{j=1}^{n} \left| \left\langle \overline{u^{q}}, \overline{v_{j}^{q}} \right\rangle - \left\langle u^{q}, v_{j}^{q} \right\rangle \right|
					= \sum_{j=1}^{n} \ER \left| \left\langle \overline{u^{q}}, \overline{v_{j}^{q}} \right\rangle - \left\langle u^{q}, v_{j}^{q} \right\rangle \right|.
			\end{equation}
		The expectation will be computed for each $j$ by conditioning on the event $\mathcal{B}_{j}^{q}$. Under this event, as derived previously, the absolute difference of the inner products is bounded by $6\delta$. Under the complementary event, which occurs with exponentially small probability in $\delta$ and $k$, the distance is trivially bounded by $5$. Thus
			\begin{align*}
				\ER \left| \left\langle \overline{u^{q}}, \overline{v_{j}^{q}} \right\rangle - \left\langle u^{q}, v_{j}^{q} \right\rangle \right| &= \ER\left[ \left| \left\langle \overline{u^{q}}, \overline{v_{j}^{q}} \right\rangle - \left\langle u^{q}, v_{j}^{q} \right\rangle \right| \vert \mathcal{B}_{j}^{q}\right] \PR\left(\mathcal{B}_{j}^{q}\right) \notag\\
					&\hspace{1cm}+ \ER\left[ \left| \left\langle \overline{u^{q}}, \overline{v_{j}^{q}} \right\rangle - \left\langle u^{q}, v_{j}^{q} \right\rangle \right| \vert \overline{\mathcal{B}_{j}^{q}}\right] \PR\left(\overline{\mathcal{B}_{j}^{q}}\right) \notag\\
					&\le 6\delta + 40 e^{-k(\delta^2 / 2 - \delta^{3}/3)/2} \notag\\
					&\le 46\delta.
			\end{align*}
		It follows that
			\begin{align*}
				\ER \norm{X^{q} - \overline{X^q}}_{1} = \sum_{j=1}^{n} \ER \left| \left\langle \overline{u^{q}}, \overline{v_{j}^{q}} \right\rangle - \left\langle u^{q}, v_{j}^{q} \right\rangle \right| \le 46\delta n.
			\end{align*}
		This implies that for each $q$, there must exist a block vector $\overline{X^{q}}$ with no more than $(1+4/\delta)^{2(k+1)}$ communities such that
			\begin{equation*}
				\norm{X^{q} - \overline{X^{q}}}_{1} \le 46\delta n.
			\end{equation*}
		Letting $c_{q}(n) = \sqrt{1 - \frac{1}{n} + \frac{(2-m_q)^2}{n}}$ (note here that $c_{q}(n) \le m_{q} - 1$) for ease of notation, we have
			\begin{equation*}
				\norm{\frac{1}{2} m_q\alpha_{q} c_{q}(n) X^{q} - \frac{1}{2} m_q\alpha_{q} c_{q}(n) \overline{X^{q}}}_{1} \le 23 m_q\left|\alpha_{q}\right| c_{q}(n) \delta n.
			\end{equation*}

		Define $\overline{X}$ as a weighted sum of the vectors $\overline{X^{q}}$ for $q=1,\dots,l$ with the all ones vector,
			\begin{equation*}
				\overline{X} = \frac{1}{2} \log\left(\frac{p}{1-p}\right) \mathbf{1}_{n} +  \frac{1}{2} \sum_{q=1}^{l} m_q\alpha_{q} c_{q}(n) \overline{X^{q}}.
			\end{equation*}
		Then $\overline{X}$ has a number of communities exponentially dependent on $\delta$ and $l$, however, most importantly, independent of $n$. The triangle inequality implies that
			\begin{equation*}
				\norm{\nabla f(X) - \overline{X} }_{1} \le 23 \left( \sum_{q=1}^{l} m_q\left|\alpha_{q}\right| c_{q}(n) \right) \delta n.
			\end{equation*}
		Since $\tanh$ is contracting,
			\begin{equation*}
				\norm{\frac{\mathbf{1}_{n} + \tanh\left( \nabla f(X) \right)}{2} - \frac{\mathbf{1}_{n} + \tanh\left( \overline{X} \right)}{2}}_1 \le \frac{23}{2} \left( \sum_{q=1}^{l} m_q\left|\alpha_{q}\right| c_{q}(n) \right) \delta n.
			\end{equation*}
		By assumption, $X \in \mathcal{X}_{f}$ and by Theorem \ref{fixedset},
			\begin{equation*}
				\norm{X - \frac{\mathbf{1}_{n} + \tanh\left( \nabla f(X) \right)}{2} }_{1} \le 5000 C_{\alpha}^{2} n^{7/8}.
			\end{equation*}
		Letting $X^{\ast} = \frac{\mathbf{1}_{n} + \tanh\left( \overline{X} \right)}{2}$, the desired conclusion is obtained,
			\begin{equation*}
				\norm{X - X^{\ast}}_{1} \le \frac{23}{2} \left( \sum_{q=1}^{l} m_q\left|\alpha_{q}\right| c_{q}(n) \right) \delta n + 5000 C_{\alpha}^{2} n^{7/8}.
			\end{equation*}
	\end{proof}

Theorems \ref{fixedset} and \ref{poscomm} together yield the following corollary relating exponential random graphs to mixtures of block vectors.

\begin{corollary} \label{cormain}
	Let $\alpha_1, \dots, \alpha_l$ be real constants and $\delta > 0$. Then there exists a $C_{\delta} > 0$ such that for every $n \in \N$, there exists a measure $\rho$ supported on block vectors with at most $C_{\delta}$ communities such that if $X_{n}^{f}$ is a random vertex configuration with distribution defined in Equation \ref{vwcd} where the Hamiltonian is as defined in Equation \ref{scf}, then there is a coupling between $X_{n}^{f}$ and $G(n,\rho)$ that satisfies
		\begin{equation}
			\ER \norm{X_{n}^{f} - G(n,\rho)}_1 \le \delta n.
		\end{equation}
\end{corollary}

We now turn to some results concerning subgraph counting functions with positive weights. Define $\varphi_{\alpha} : [0,1] \to \R$ by
	\begin{equation}
		\varphi_{\alpha}(x) = \frac{1 + \tanh\left( \frac{1}{2} \log\left(\frac{p}{1-p}\right) + \frac{1}{2} \sum_{q=1}^{l} m_q\alpha_{q} \prod_{i=1}^{m_{q}-1} \left(1 - \frac{i}{n}\right) x^{m_q - 1} \right)}{2}.
	\end{equation}
Note that $\varphi$ is equal to the entry of the constant vector for the fixed point Equation \ref{vector}. If $\varphi_{\alpha}$ has a unique fixed point $x$, define $D_{\alpha}$ as
	\begin{equation}
		D_{\alpha} = \sup_{\substack{y \in [0,1] \\ y \neq x}} \frac{\left|\varphi_{\alpha}(y) - x\right|}{\left|y - x\right|}.
	\end{equation}

\begin{lemma} \label{ffp}
	There exists an $x \in (0,1)$ such that $x = \varphi_{\alpha}(x)$. Now assume that $\varphi_{\alpha}$ is increasing. If the solution $x$ is unique and $\varphi^{\prime}_{\alpha}(x) < 1$, then $D_{\alpha} < 1$.
\end{lemma}

	\begin{proof}
		Since $\varphi_{\alpha} : [0,1] \to (0,1)$ is continuous, there exists an $x \in (0,1)$ such that $x = \varphi_{\alpha}(x)$. Now assume that $\varphi_{\alpha}$ is increasing, $x$ is a unique fixed point and $\varphi^{\prime}_{\alpha}(x) < 1$. Consider the function $h(y) = \varphi_{\alpha}(y) - x$. For $y < x$, $\varphi_{\alpha}(y) \le x$ and
			\begin{equation*}
				0 \ge h(y) = \varphi_{\alpha}(y) - x > y - x.
			\end{equation*}
		Hence
			\begin{equation*}
				0 \le \frac{\varphi_{\alpha}(y) - x}{y - x} < 1.
			\end{equation*}
		On the other hand, for $y > x$, $\varphi(y) \ge x$ and
			\begin{equation*}
				0 \le h(y) = \varphi_{\alpha}(y) - x < y - x.
			\end{equation*}
		Therefore again
			\begin{equation*}
				0 \le \frac{\varphi_{\alpha}(y) - x}{y - x} < 1.
			\end{equation*}
		Combining these inequalities for $y < x$ and $y > x$ with the fact that $\varphi^{\prime}_{\alpha}(x) < 1$, it follows that $D_{\alpha} < 1$. Note also that if $\varphi_{\alpha}$ is not constant, then $D_{\alpha} > 0$.
	\end{proof}

\begin{theorem} \label{fpsolution}
	Let $n$ be a positive integer. Let $f$ be a subgraph counting function of the form \ref{scf}. Assume that $\alpha_{q} \ge 0$ for all $q \ge 1$, $x = \varphi_{\alpha}(x)$ has a unique solution $x$, and $D_{\alpha} < 1$. Then for any $X \in \mathcal{X}_{f}$ and any $0 < \varepsilon < 1$,
		\begin{equation}
			\norm{X - x\mathbf{1}_{n}}_{1} \leq \varepsilon n + 10000 C_{\alpha}^{2} \varepsilon^{\frac{\log C_{\alpha}}{\log D_{\alpha}}} n^{7/8}.
		\end{equation}
	Furthermore, for any constants $C_{\alpha}$ and $D_{\alpha}$, there exist constants $c > 0$ and $0 < \eta < 1/8$ such that
		\begin{equation}
			\norm{X - x\mathbf{1}_{n}}_{1} \le c n^{1-\eta}.
		\end{equation}
\end{theorem}

\begin{proof}
It will first be shown that if the weights $\alpha_{q}$ are positive for $q \ge 1$ and $x$ is the unique fixed point of $\varphi_{\alpha}$, then the fixed point Equation \ref{vector} has a unique solution $X = x \mathbf{1}_{n}$. Let $X$ be a solution to this vector equation. For ease of notation let $a$ and $b$ be defined as $a = \min X$ and $b = \max X$. Since the weights are positive,
			\begin{equation*}
				\partial_{j} f(X) \le \frac{1}{2} \left( \log\left(\frac{p}{1-p} \right) + \sum_{q=1}^{l} m_q\alpha_q \prod_{i=1}^{m_q -1} \left(1 - \frac{i}{n} \right) b^{m_q - 1} \right).
			\end{equation*}
		Taking the maximum of the right hand side and then the left hand side in Equation \ref{vector}, we have
			\begin{equation*}
				b \le \frac{1 + \tanh\left( \frac{1}{2} \left( \log\left(\frac{p}{1-p} \right) + \sum_{q=1}^{l} m_q\alpha_q \prod_{i=1}^{m_q - 1} \left(1 - \frac{i}{n} \right) b^{m_q - 1} \right) \right)}{2}.
			\end{equation*}
		Similarly for $a$, since the weights are positive, a parallel reasoning gives
			\begin{equation*}
				a \ge \frac{1 + \tanh\left( \frac{1}{2} \left( \log\left(\frac{p}{1-p} \right) + \sum_{q=1}^{l} m_q\alpha_q \prod_{i=1}^{m_q -1} \left(1 - \frac{i}{n} \right) a^{m_q - 1} \right) \right)}{2}.
			\end{equation*}
		By assumption, $\varphi_{\alpha}$ has a unique solution $x$. The preceding two inequalities thus give that $a \ge x$ and $b \le x$. Since $b$ is the maximum of the solution $X$ and $a$ is the minimum, this implies that $X = x \mathbf{1}_{n}$ is the unique solution to the fixed point Equation \ref{vector}.

Next we show that for all $X \in \mathcal{X}_{f}$, $X$ is close to the constant solution $x \mathbf{1}_{n}$ in $1$-norm distance. Let $X \in \mathcal{X}_{f}$ and $0 < \varepsilon < 1$. Define a sequence of functions $\{ \varphi_{i}(x) \}_{i=1}^{\infty}$ such that $\varphi_{1}(x) = \varphi_{\alpha}(x)$ and $\varphi_{i+1}(x) = \varphi_{\alpha} \left(\varphi_{i}(x)\right)$ for $i \ge 1$. For all $y \in [0,1]$,
			\begin{equation*}
				\left| \varphi_{1}(y) - x \right| \le D_{\alpha} \left|y - x\right|.
			\end{equation*}
		Iterating this preceding inequality yields
			\begin{equation*}
				\left| \varphi_{i}(y) - x \right| \le D_{\alpha}^{i} \left| y - x \right|
			\end{equation*}
		for all $i \in \N$. Set $k = \left\lceil \log_{D_{\alpha}} \varepsilon \right\rceil$. Then for $i = k$ in particular,
			\begin{equation} \label{lambdaineq}
				\left| \varphi_{k}(y) - x \right| \le D_{\alpha}^{k} \left| y - x \right| < \varepsilon.
			\end{equation}
		Let $\Phi : \overline{\mathcal{C}}_{n} \to \overline{\mathcal{C}}_{n}$ be the function
			\begin{equation*}
				\Phi(X) = \frac{\mathbf{1}_{n} + \tanh\left( \nabla f(X) \right)}{2}.
			\end{equation*}
		Define a sequence of vectors in $\R^{n}$ as $Y_{0} = X$ and recursively $Y_{i+1} = \Phi(Y_{i})$ for $i \ge 0$. Since $\alpha_{q} \ge 0$ for $q \ge 1$,
			\begin{align*}
				\min Y_1 &= \min \Phi(X) = \min \frac{\mathbf{1}_{n} + \tanh\left( \nabla f(X) \right)}{2} \notag \\
				&\ge \min \frac{\mathbf{1}_{n} + \tanh\left( \nabla f((\min X) \mathbf{1}_{n}) \right)}{2} = \varphi_{\alpha}\left( \min X \right).
			\end{align*}
		Analogously,
			\begin{align*}
				\max Y_1 &= \max \Phi(X) = \max \frac{\mathbf{1}_{n} + \tanh\left( \nabla f(X) \right)}{2} \notag \\
				&\le \max \frac{\mathbf{1}_{n} + \tanh\left( \nabla f((\max X) \mathbf{1}_{n}) \right)}{2} = \varphi_{\alpha}\left( \max X \right).
			\end{align*}
		Iterating these inequalities for $\min$ and $\max$ gives
			\begin{equation*}
				\varphi_{k} (\min X) \le \min Y_{k} \le \max Y_{k} \le \varphi_{k} (\max X).
			\end{equation*}
		Since Inequality \ref{lambdaineq} holds for all $y \in [0,1]$, it must hold for both $\min X$ and $\max X$, which says that
			\begin{equation*}
				x - \varepsilon < \min Y_{k} \le \max Y_{k} < x + \varepsilon.
			\end{equation*}
		In other words, all entries of the vector $Y_{k}$ lie within $\varepsilon$ of $x$ in $1$-norm. Hence
			\begin{equation*}
				\norm{Y_{k} - x\mathbf{1}_{n}}_{1} < \varepsilon n.
			\end{equation*}

		We have established that $Y_{k}$ is close to the constant solution $x\mathbf{1}_{n}$. In order to show that $X$ is close to the constant solution, we need to show that $X$ is close to $Y_{k}$. By Lemma \ref{normdistpreserved}, for any $U, V \in \overline{\mathcal{C}}_{n}$,
			\begin{equation*}
				\norm{\Phi(U) - \Phi(V)}_{1} \le C_{\alpha} \norm{U - V}_{1}.
			\end{equation*}
		Specifically, for consecutive vectors $Y_{i}$ and $Y_{i-1}$,
			\begin{equation*}
				\norm{Y_{i} - Y_{i-1}}_{1} = \norm{\Phi(Y_{i-1}) - \Phi(Y_{i-2})}_{1} \le C_{\alpha} \norm{Y_{i-1} - Y_{i-2}}_{1}.
			\end{equation*}
		Consecutive iterations yield
			\begin{equation*}
				\norm{Y_{i} - Y_{i-1}}_{1} \le C_{\alpha}^{i-1} \norm{Y_{1} - Y_{0}}_{1} = C_{\alpha}^{i-1} \norm{\Phi(X) - X}_{1}.
			\end{equation*}
		Using this bound, we can bound the distance between $X$ and $Y_{k}$ as
			\begin{align*}
				\norm{X - Y_{k}}_{1} &= \norm{\sum_{i=1}^{k} \left( Y_{i} - Y_{i-1} \right)}_{1} \le \sum_{i=1}^{k} \norm{Y_{i} - Y_{i-1}}_{1} \notag \\
					&\le \sum_{i=0}^{k-1} C_{\alpha}^{i} \norm{\Phi(X) - X}_{1} \le 2C_{\alpha}^{k-1} \norm{\Phi(X) - X}_{1}.
			\end{align*}
		Compiling the above results and applying Theorem \ref{fixedset} allows us to bound the distance between $X$ and $x\mathbf{1}_{n}$ as follows:
			\begin{align} \label{epsdist}
				\norm{X - x\mathbf{1}_{n}}_{1} &\le \norm{X - Y_{k}}_{1} + \norm{Y_{k} - x\mathbf{1}_{n}}_{1} \notag \\
					&< 2C_{\alpha}^{k-1} \norm{\Phi(X) - X}_{1} + \varepsilon n \notag \\
					&\le 10000 \varepsilon^{\frac{\log C_{\alpha}}{\log D_{\alpha}}} C_{\alpha}^{2} n ^{7/8} + \varepsilon n.
			\end{align}
		The order conclusion on $n$ is deduced by minimizing the last display over $0 < \varepsilon < 1$, which occurs when
			\begin{equation*}
				\varepsilon = \left( \frac{\log{\frac{1}{D_{\alpha}}}}{10000 C_{\alpha}^{2} \log{C_{\alpha}}} n^{1/8} \right)^{\frac{\log{D_{\alpha}}}{\log{\frac{C_{\alpha}}{D_{\alpha}}}}}.
			\end{equation*}
	\end{proof}

\begin{remark} \label{asymptend}
	Note that the exponent of $n$ in the first term of Display \ref{epsdist} is
	\[ \frac{7}{8} + \frac{\log{C_{\alpha}}}{8\log\frac{C_{\alpha}}{D_{\alpha}}} = \frac{7}{8} +  \frac{\log{C_{\alpha}}}{8\left(\log{C_{\alpha}} - \log{D_{\alpha}}\right)}, \]
and the exponent of $n$ in the second term is
	\[ 1 + \frac{\log{D_{\alpha}}}{8\log\frac{C_{\alpha}}{D_{\alpha}}} = 1 + \frac{\log{D_{\alpha}}}{8\left(\log{C_{\alpha}} - \log{D_{\alpha}}\right)}. \]
For a fixed $C_{\alpha}$, as $D_{\alpha}$ approaches $0$, the order of both terms tend to $n^{7/8}$; while, as $D_{\alpha}$ limits to $1$ from the left, the order of both terms tend to $n$. The rate at which the order approaches these values is dependent on $C_{\alpha}$. Under the assumption that $C_{\alpha} \le \frac{1}{D_{\alpha}}$, we will show explicitly that $\norm{X - x\mathbf{1}_{n}}_{1} = o(n)$. From the proof of Theorem \ref{fpsolution},
		\begin{equation*}
			\norm{X - x\mathbf{1}_{n}}_{1} \le 10000 C_{\alpha} C_{\alpha}^{k} n^{7/8} + D_{\alpha}^{k} n
		\end{equation*}
	for any $k \in \N$. Let $k$ be of the form $k = \left\lceil \log_{D_{\alpha}}{n^{-m}} \right\rceil$ where $m > 0$. Then $D_{\alpha}^{k} \le n^{-m}$ and, since $C_{\alpha} \le \frac{1}{D_{\alpha}}$,
		\begin{equation*}
			C_{\alpha}^{k} \le C_{\alpha} C_{\alpha}^{\log_{D_{\alpha}}{n^{-m}}} \le C_{\alpha} n^{m}.
		\end{equation*}
	Therefore
		\begin{equation*}
			\norm{X - x\mathbf{1}_{n}}_{1} \le 10000 C_{\alpha}^{2} n^{m + 7/8} + n^{1-m},
		\end{equation*}
	and minimizing this over $m$ yields
		\begin{equation*}
			\norm{X - x\mathbf{1}_{n}}_{1} \le \left(10000 C_{\alpha}^{2} + 1\right) n^{15/16}.
		\end{equation*}
\end{remark}

\begin{remark}
	The unique solution of $x = \varphi_{\alpha}(x)$ is close to the solution to the equation
		\begin{equation*}
			x = \frac{1 + \tanh\left( \frac{1}{2} \left( \log\left(\frac{p}{1-p} \right) + \sum\limits_{q=1}^{l} m_q\alpha_q x^{m_q - 1} \right) \right)}{2},
		\end{equation*}
	thereby eliminating the $n$ dependence asymptotically. This condition on the positive weights is commonly referred to as the \textit{high temperature regime}. See Figure \ref{fixed}.
\end{remark}

\begin{figure}[t!]
	\centering
	\includegraphics[clip=true, height=5.5in, angle=-90]{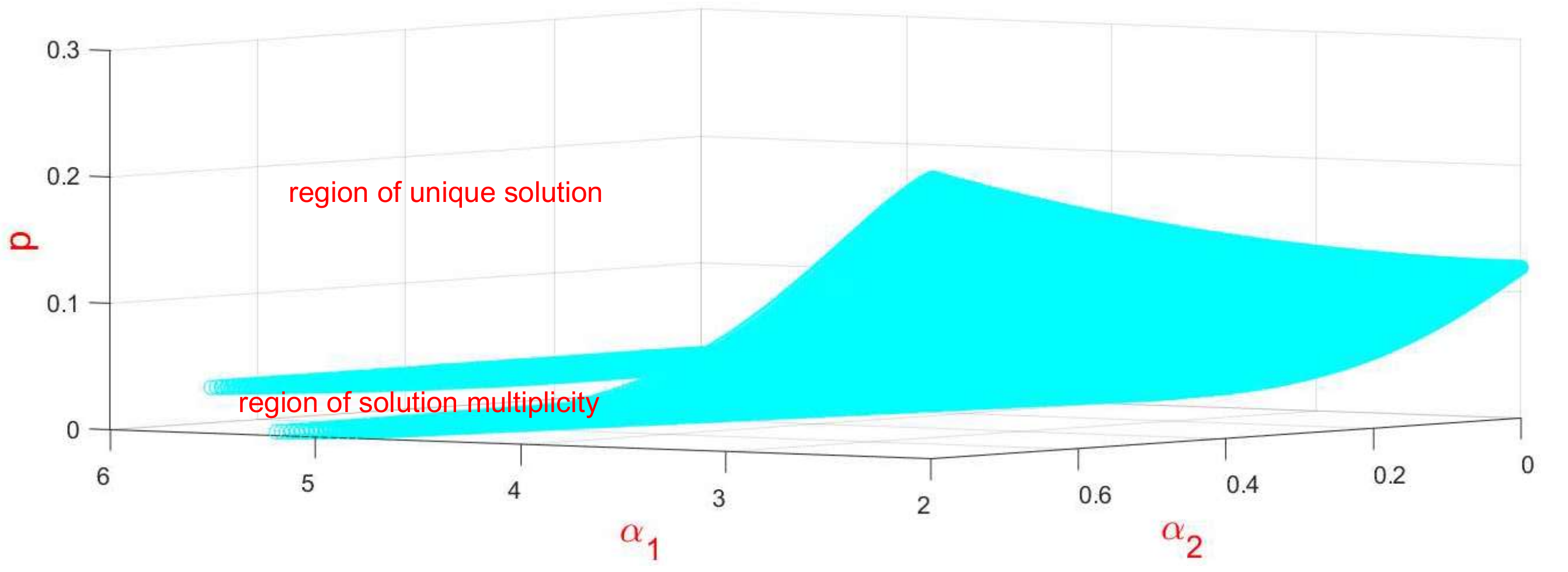}
	\caption{The parameter space as separated by the region with unique vs. multiple solutions to $x = \varphi_{\alpha}(x)$. The asymptotic illustration is for the vertex-weighted edge-triangle model.}
    \label{fixed}
\end{figure}

We now investigate a specific example for a family of subgraph counting functions. It demonstrates that the preceding theorem can give meaningful information with respect to sparse graphs and that the condition $C_{\alpha} \le \frac{1}{D_{\alpha}}$ from Remark \ref{asymptend} is achieved for some substantively interesting models.

\begin{example}
	Consider the subgraph counting function
		\begin{equation*}
			f(X) = \log\left(\frac{p}{1-p}\right)||X||_1+\frac{\alpha}{n} \sum_{i \neq j} X_{i} X_{j},
		\end{equation*}
	where vertex weights take value $0$ with probability $1-p$ and $1$ with probability $p$ and $\alpha \geq 0$.
	
For notational convenience denote $\log \left(p/(1-p)\right)$ by $\gamma$. Consider $p=p(n)=n^{-d}$ for some $d>0$ and $\left|\gamma\right|/200 \leq \alpha \leq \left|\gamma\right|/100$. Ignoring the second term in $f$, this resembles the $G(n, p)$ model of Gilbert \cite{G}. Hence the expected number of vertices will be at least of order $np$. With $p$ of the preceding form, $\gamma \approx -\log{n}$ and $\alpha \approx \log{n}$. For $n$ large enough, $\varphi_{\alpha}(x)$ has a unique fixed point; denote it by $\overline{x}$. Calculating $D_{\alpha}$, we have
	\begin{align*}
		D_{\alpha} &\le \max_{x \in [0,1]} \left| \varphi_{\alpha}^{\prime}(x) \right| \\
			&= \max_{x \in [0,1]} \left| \frac{\alpha \left(1- \frac{1}{n}\right)}{2 \cosh^2 \left(\frac{1}{2}\gamma+\alpha \left(1- \frac{1}{n}\right)x \right)} \right| \\
			&\lesssim \left|\gamma\right| e^{-\left|\gamma\right|/2}.
	\end{align*}
Since $\alpha \to \infty$, for $n$ sufficiently large $D_{\alpha} < 1$ and $\left|\log D_{\alpha}\right| \gtrsim \left|\gamma\right|$. This allows us to utilize Theorem \ref{fpsolution} to say that for all $X \in \mathcal{X}_{f}$,
	\begin{equation}  \label{normbound1}
		\norm{X - \overline{x}\mathbf{1}_{n}}_{1} \le 10000 \varepsilon^{\frac{\log C_{\alpha}}{\log D_{\alpha}}} C_{\alpha}^{2} n ^{7/8} + \varepsilon n.
	\end{equation}
From Equation \ref{calpha}, $C_{\alpha} \approx \left|\gamma\right|$. Note that in the limit $\left|\gamma\right|^2 e^{-\left|\gamma\right|/2} \le 1$, we conclude that the condition $C_{\alpha} \le \frac{1}{D_{\alpha}}$ is satisfied. Examining the exponent on $\varepsilon$ in the first term, we find
	\[ \left| \frac{\log C_{\alpha}}{\log D_{\alpha}} \right| \lesssim \frac{\log{\log{n}}}{d \log{n}}. \]
Since $n$ is finite and fixed, let $\varepsilon = n^{-1}$, then
\[\varepsilon^{\frac{\log C_{\alpha}}{\log D_{\alpha}}} \approx \left( \log{n} \right)^{1/d}.\]
To yield a meaningful bound in Equation \ref{normbound1}, we want the error term to be smaller than the typical number of vertices. Ignoring lower order terms such as logarithms, we require
	\[np \gtrsim n^{7/8},  \]
which forces the condition that $p \gtrsim n^{-1/8}$. For $p$ in this region, the majority of the probability mass in the model lies on sparse graphs. Furthermore, by Corollary \ref{cormain}, there exists a constant $\overline{p}$ and a coupling between $G(n, \overline{p})$ and $X_{n}^{f}$ such that
	\[ \ER\norm{X_{n}^{f}-G(n,\overline{p})}_{1} = o(np). \]
\end{example}

Theorem \ref{fpsolution} shows that for positive weights, any $X \in \mathcal{X}_{f}$ is at most $cn^{1-\eta}$ distance away from the fixed point of $\Phi$ for some constants $c$ and $0 < \eta < 1/8$. This distance can be refined to $cn^{7/8}$ if one assumes that the weights $\alpha_q$ are sufficiently small (either positive or negative). Recall the function $\Phi : \overline{\mathcal{C}}_{n} \to \overline{\mathcal{C}}_{n}$ defined in the proof of this theorem,
	\begin{equation*}
		\Phi(X) = \frac{\mathbf{1}_{n} + \tanh\left( \nabla f(X) \right)}{2}.
	\end{equation*}

\begin{theorem} \label{smallweights}
	Let $f$ be a subgraph counting function of the form \ref{scf} and let
		\begin{equation*}
			J_{\alpha} = \frac{1}{4} \sum_{q=1}^{l} \left|\alpha_q\right| m_q(m_q - 1).
		\end{equation*}
	If $J_{\alpha} < 1$, then the constant solution $X_{c}$ is the only solution to the fixed point equation $X = \Phi(X)$ and, for any $X \in \mathcal{X}_{f}$,
		\begin{equation*}
			\norm{X - X_{c}}_{1} \le \frac{5000C_{\alpha}^{2}}{1-J_{\alpha}} n^{7/8}.
		\end{equation*}
\end{theorem}

	\begin{proof}
		By Lemma \ref{normdistpreserved}, for $X,Y \in \overline{\mathcal{C}}_{n}$,
			\begin{equation*}
				\norm{\nabla f(X) - \nabla f(Y)}_{1} \le \frac{1}{2} \sum_{q=1}^{l} \left|\alpha_q\right| m_q(m_q - 1) \norm{X - Y}_{1}.
			\end{equation*}
		Therefore
			\begin{align*}
				\norm{\Phi(X) - \Phi(Y)}_{1} &= \norm{\frac{\mathbf{1}_{n} + \tanh(\nabla f(X))}{2} - \frac{\mathbf{1}_{n} + \tanh(\nabla f(Y))}{2}} \notag \\
					&\le \frac{1}{2} \norm{\nabla f(X) - \nabla f(Y)}_{1} \notag \\
					&\le \frac{1}{4} \sum_{q=1}^{l} \left|\alpha_q\right| m_q(m_q - 1) \norm{X - Y}_{1} \notag \\
					&= J_{\alpha} \norm{X - Y}_{1}.
			\end{align*}
		Since $J_{\alpha} < 1$, $\Phi$ is a contraction mapping, further implying that $\Phi$ has a unique fixed point: the constant solution $X_c=x\mathbf{1}_n$ provided by Lemma \ref{ffp} and Theorem \ref{fpsolution}. Now let $X \in \mathcal{X}_{f}$. Then
			\begin{align*}
				\norm{X - X_{c}}_{1} &= \norm{X - \Phi(X) + \Phi(X) - \Phi(X_c) + \Phi(X_c) - X_c}_{1} \notag \\
					&\le \norm{X - \Phi(X)}_{1} + \norm{\Phi(X) - \Phi(X_c)}_{1} + \norm{\Phi(X_c) - X_c}_{1} \notag \\
					&\le \norm{X - \Phi(X)}_{1} + J_{\alpha}\norm{X - X_{c}}_{1}.
			\end{align*}
		From the above inequality, one finds that
			\begin{equation*}
				\norm{X - X_{c}}_{1} \le \frac{\norm{X - \Phi(X)}_{1}}{1-J_{\alpha}} \le \frac{5000C_{\alpha}^{2}}{1-J_{\alpha}} n^{7/8}.
			\end{equation*}
	\end{proof}

\section{Triangle Case}

Let us take a closer look at the special case of the subgraph counting function consisting solely of triangles
	\begin{equation} \label{tri}
		f(X) = \frac{\alpha}{n^2} \sum_{i \neq k \neq j} X_i X_k X_j.
	\end{equation}
Eldan and Gross showed in Theorem 21 of \cite{EG2} that for a subgraph counting function that counts triangles in simple graphs, there exists a constant $\alpha_0 < 0$ such that for all $\alpha<\alpha_0$ there is a non-constant solution in the form of a stochastic block matrix consisting of two communities. Furthermore, they determined the asymptotic tendency of the edge presence probability as $\alpha \to -\infty$.

Without loss of generality, take $n$ even. (The odd case follows a similar line of reasoning.) Since we are, implicitly, placing equal probability on vertex weights being $0$ and $1$ (cf. Equation \ref{single}), some structure of symmetry is expected. We assume that there is a solution to the fixed point equation
	\begin{equation} \label{fixedpt}
		X = \frac{\mathbf{1}_{n} + \tanh(\nabla f(X))}{2}
	\end{equation}
in the form of a vector with the first $n/2$ entries taking value $a \in [0,1]$ and second $n/2$ entries taking value $b \in [0,1]$,
\[ X = (a,a,\dots,a,b,b,\dots,b). \]
Recall the definition of the discrete derivative. Note that for the subgraph counting function defined above in Equation \ref{tri}, the discrete derivative at the $j$th coordinate may be written as
	\[ \partial_j f(X) = \frac{3\alpha}{2n^2} \sum_{i \neq k \neq j} X_i X_k. \]
From here, we examine the possibilities for the product $X_i X_k$ in the sum. There are three types of terms in this sum since there are two different entries in the vector $X$. Working out the various combinations gives
	\[ \partial_j f(X) = \frac{3\alpha}{n^2} \left( {\frac{n}{2} - 1 \choose 2} a^2 + {\frac{n}{2} \choose 2} b^2 + \left(\frac{n}{2} - 1\right) \frac{n}{2} ab \right).\]
Writing this more compactly, it follows that if $X_j = a$, then
	\[ \partial_j f(X) = \frac{3\alpha}{2n^2} \left( \frac{n}{2} - 1 \right) \left(\frac{n}{2}(a+b)^2 - 2a^2 \right), \]
and similarly, if $X_j = b$, then
	\[ \partial_j f(X) = \frac{3\alpha}{2n^2} \left( \frac{n}{2} - 1 \right) \left(\frac{n}{2}(a+b)^2 - 2b^2 \right). \]
From here we obtain a system of two equations in $a$ and $b$ that must be solved simultaneously:
	\begin{equation*}
		a = \frac{1 + \tanh\left( \frac{3\alpha}{2n^2} \left( \frac{n}{2} - 1 \right) \left(\frac{n}{2}(a+b)^2 - 2a^2 \right) \right)}{2}
	\end{equation*}
and
	\begin{equation*}
		b = \frac{1 + \tanh\left( \frac{3\alpha}{2n^2} \left( \frac{n}{2} - 1 \right) \left(\frac{n}{2}(a+b)^2 - 2b^2 \right) \right)}{2}.
	\end{equation*}
Instead of considering the previous two equations for $\alpha < 0$, for simplicity, we consider the equations
	\begin{equation} \label{a-eq}
		1 - 2a = \tanh\left( \frac{\alpha}{2n^2} \left( \frac{n}{2} - 1 \right) \left(\frac{n}{2}(a+b)^2 - 2a^2 \right) \right)
	\end{equation}
and
	\begin{equation} \label{b-eq}
		1 - 2b = \tanh\left( \frac{\alpha}{2n^2} \left( \frac{n}{2} - 1 \right) \left(\frac{n}{2}(a+b)^2 - 2b^2 \right) \right)
	\end{equation}
for $\alpha > 0$, where we assimilated the factor of $3$ inside $\alpha$.

Let $n$ be an even integer such that $n > 4$. Fix $b = 1$ in Equation \ref{a-eq}. The left hand side of Equation \ref{a-eq} is decreasing from $1$ to $-1$ while the right hand side is increasing with values between $0$ and $1$. Thus the left side and right side intersect in a unique $a_1 \in (0,1)$ that solves Equation \ref{a-eq} for $b = 1$. Now let $b = 0$ in the same equation. By similar reasoning, there exists a unique $a_2 \in (0,1)$ that solves Equation \ref{a-eq} for $b = 0$ with $a_1 < a_2$. For each $b \in [0,1]$, there is a unique $a$ such that the pair $(a,b)$ is a solution to Equation \ref{a-eq}. Let $g:[0,1] \to [a_1,a_2]$ be the function that produces a unique $a$ for each $b$, that is, $g(b) = a$ and the pair $(a,b)$ is a solution to Equation \ref{a-eq}. Similarly, we can consider a fixed $a \in [0,1]$ in Equation \ref{b-eq} to obtain a unique $b \in (0,1)$ such that the pair $(a,b)$ is a solution to Equation \ref{b-eq}. This also yields a corresponding $b_1$ and $b_2$ when fixing $a = 1$ and $a = 0$, respectively. Denote by $h:[0,1] \to [b_1,b_2]$ the function that produces a unique $b$ for each $a \in [0,1]$ such that $h(a) = b$ and $(a,b)$ is a solution to Equation \ref{b-eq}. Both $g$ and $h$ are decreasing on their respective domains. Note that if the functions $g$ and $h$ intersect at some point $(a,b)$, then the point $(a,b)$ is a solution to Equations \ref{a-eq} and \ref{b-eq} simultaneously; therefore, the pair $(a,b)$ forms a solution to the fixed point equation. 

Consider now the behavior of $a_1$ and $a_2$ as $\alpha$ diverges to infinity. We first show that $a_2$ tends to $0$ as $\alpha \to \infty$ with a rate of decay controlled by a polynomial. This implies that $a_1$ also approaches $0$ as $\alpha$ diverges. We also prove a corresponding result that controls the rate at which $a_1$ approaches $0$, showing that the decay is bounded below by an exponential. Using the symmetry afforded by Equations \ref{a-eq} and \ref{b-eq}, similar results hold for $b_1$ and $b_2$.

For the proof of the following Lemma \ref{a2}, we require the Lambert $W$ function. The Lambert $W$ function is a set of functions satisfying the functional equation
	\[ W(z)e^{W(z)} = z \]
for any $z \in \mathbb{C}$, and may be equivalently written as $W(ze^z) = z$. If we restrict our attention to the reals, there are two branches of this function, denoted by $W_{-1}$ and $W_{0}$, with the latter known as the principal branch. We utilize the branch $W_0$ in the following proof where $W_{0} : [-1/e,\infty) \to [-1,\infty)$. The principal branch satisfies the inequality
	\begin{equation} \label{lambertineq}
		\log{x} - \log{\log{x}} + \frac{\log{\log{x}}}{2\log{x}} \le W_{0}(x) \le \log{x} - \log{\log{x}} + \frac{e}{e-1} \frac{\log{\log{x}}}{\log{x}}
	\end{equation}
for all $x \ge e$ as shown in \cite{HH}. For more information on the Lambert $W$ function, see \cite{CGHJK}.

\begin{lemma} \label{a2}
Take $\alpha$ sufficiently large. Let $a_2$ be defined as above. Then
		\begin{equation} \label{a2lambert}
			a_2 < \sqrt{\frac{n^2 \log\left( \frac{4\alpha}{n^2} {\frac{n}{2} -1 \choose 2} \right)}{4\alpha{\frac{n}{2} -1 \choose 2} }}.
		\end{equation}
\end{lemma}

	\begin{proof}
		Recall that $a_2$ is the unique solution to
			\[ 1-2a_2 = \tanh\left( \frac{\alpha}{2n^2} \left( \frac{n}{2} - 1 \right) \left( \frac{n}{2} - 2 \right) a_{2}^{2} \right). \]
		Let $N$ be defined as
			\[ N = \frac{1}{2n^2} \left( \frac{n}{2} - 1 \right) \left( \frac{n}{2} - 2 \right).  \]
		Then the preceding equation becomes
			\[ 2a_2 = 1 - \tanh(N\alpha a_{2}^2). \]
		Converting $\tanh$ to exponentials yields
			\[ a_2 = \frac{1}{e^{2N\alpha a_2^2} + 1} < e^{-2N\alpha a_2^2}. \]
		Letting $a = a_2^2$, we wish to solve the inequality
			\[ \sqrt{a} e^{2N\alpha a} < 1. \]
		Squaring both sides and multiplying by $4N\alpha$ gives
			\[ 4N\alpha a e^{4N\alpha a} < 4N\alpha \]
		and, since $4N\alpha > e^{e/(e-1)}$ for $\alpha$ large eventually, we use the function $W_0$ to obtain
			\[ a < \frac{W_{0}(4N\alpha)}{4N\alpha} < \frac{\log\left( 4N\alpha \right)}{4N\alpha}, \]
		where the final inequality comes from using the bound in Equation \ref{lambertineq}. Converting back to $a_2$ and $n$, the desired Inequality \ref{a2lambert} is derived.
	\end{proof}

We now show that the decay of $a_1$ to $0$ is slower than a certain exponential.

\begin{lemma} \label{a1}
Take $\alpha$ sufficiently large. Let $a_1$ be defined as above. Then $a_1 > e^{-c\alpha}$ for all $\displaystyle c > \frac{n-2}{n}$.
\end{lemma}

 	\begin{proof}
 		We prove this result in a similar way to the previous Lemma \ref{a2}. Let $c$ be such that
 			\[ c > \frac{n-2}{n}. \]
 		Recall that $a_1$ is the unique solution to
 			\[ 1 - 2a = \tanh\left( \frac{\alpha}{2n^2} \left( \frac{n}{2} - 1 \right) \left(\frac{n}{2}(a+1)^2 -2a^2 \right) \right). \]
 		We will show that if $a = e^{-c\alpha}$ in the previous display, then the left hand side must be greater than the right hand side. This implies that $a_1 > e^{-c\alpha}$. Consider the limit
 			\[ \lim_{\alpha \to \infty} \frac{2e^{-c\alpha}}{1 - \tanh\left( \frac{\alpha}{2n^2} \left( \frac{n}{2} - 1 \right) \left(\frac{n}{2}(e^{-c\alpha}+1)^2 -2e^{-2c\alpha} \right) \right)} \]
 		and let $y(\alpha)$ denote the expression inside of $\tanh$. Note that $y(\alpha) \le \alpha (n-2)/(2n)$. Using L'H\^{o}pital's rule we find
 			\begin{align*}
 				\lim_{\alpha \to \infty} \frac{2e^{-c\alpha}}{1 - \tanh\left(y(\alpha) \right)}
 				&= \lim_{\alpha \to \infty} 2c\left( \frac{dy}{d\alpha} \right)^{-1} e^{-c\alpha} \cosh^2 \left( y(\alpha) \right) \\
 				&= \lim_{\alpha \to \infty} \frac{c}{2}\left( \frac{dy}{d\alpha} \right)^{-1} e^{-c\alpha} \left( e^{-2 y(\alpha)} + e^{2 y(\alpha)} + 2 \right) \\
 				&= 0.
 			\end{align*}
 		The last line can be justified as follows. The expression $dy/d\alpha$ limits to a nonzero constant. Also, since $c > (n-2)/n$ and $y(\alpha) \le \alpha (n-2)/(2n)$, it follows that $c \alpha > 2y(\alpha)$. 
 		This implies that $a_1 > e^{-c\alpha}$.
 	\end{proof}

Lemma \ref{a2} proves that as $\alpha \to \infty$, both $a_1$ and $a_2$ approach zero since $a_1 < a_2$. Lemma \ref{a1} gives that the speed with which $a_1$ tends to zero is bounded below. 
Taking the limit of $n$ in Equations \ref{a-eq} and \ref{b-eq} yields
	\[ 1-2a = \tanh\left(\frac{\alpha (a+b)^2}{8}\right) \]
and
	\[ 1-2b = \tanh\left(\frac{\alpha (a+b)^2}{8}\right) \]
by the continuity of $\tanh$. Thus, for $n$ large enough, the only solution to this system is given by the constant solution $a = b$. By Lemma 33 of \cite{EG2}, the equation
	\[ 1-2a = \tanh\left( \frac{\alpha a^2}{2} \right) \]
has a unique solution $a \in (0,1)$ for all $\alpha \in \R$. Therefore for this special asymptotic scenario of vertex-weighted graphs, unlike the simple graphs counterpart, the only solution to the fixed point Equation \ref{fixedpt} is the constant vector, which approaches the all zeros vector as $\alpha \to \infty$ by Lemma \ref{a2}.

\begin{figure}[tp]
\centering
\includegraphics[width=.33333\textwidth]{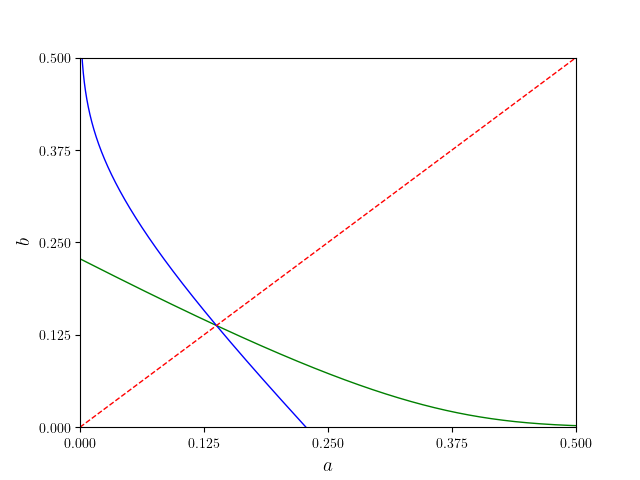}\hfill
\includegraphics[width=.33333\textwidth]{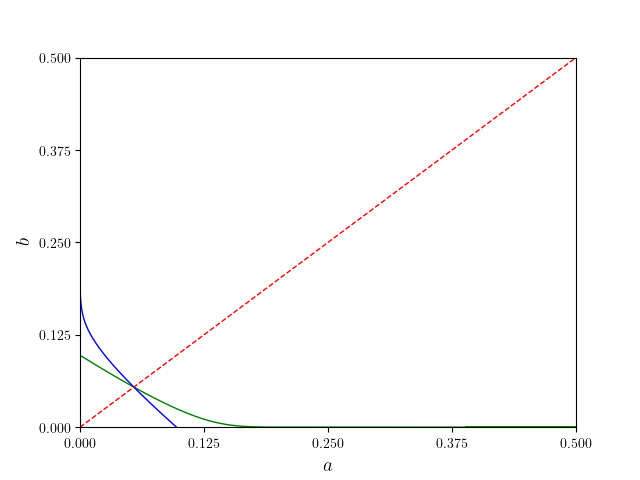}\hfill
\includegraphics[width=.33333\textwidth]{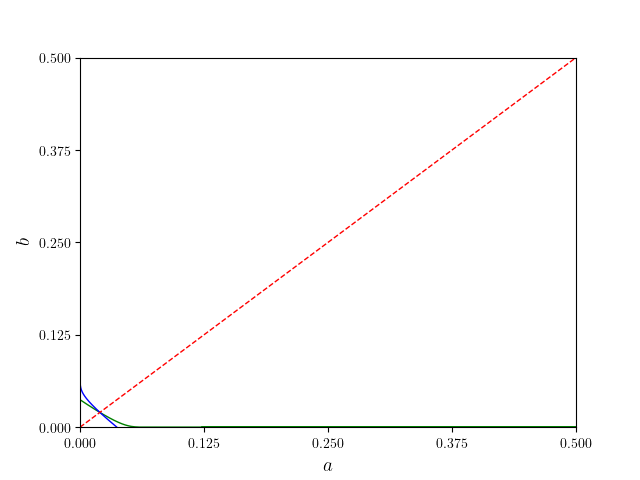}
\caption{These figures simulate the movement of the solution to the system of Equations \ref{a-eq} and \ref{b-eq} toward the origin.}
\label{SysSol}
\end{figure}

Figure \ref{SysSol} simulates the movement of the solution toward the origin as $\alpha$ diverges to $\infty$. The value of $n$ was taken to be sufficiently large and even. From left to right in Figure \ref{SysSol}, the values of $\alpha$ are $100$, $1000$, and $10000$.

\section{Other Subgraph Counting Functions}

We now investigate another type of subgraph counting function on the space of $\{0,1\}$ vertex-weighted exponential random graphs. Consider the Hamiltonian $f: \mathcal{C}_n \to \R$ such that
	\begin{equation} \label{sums}
		f(X) = \frac{\alpha}{n} \sum_{i \neq j} \left(X_i + X_j\right) + \log\left(\frac{p}{1-p}\right) \norm{X}_1,
	\end{equation}
where $\alpha \in \R$. We interpret this function in the context of social networks as follows. Fix $i \neq j$. Suppose that vertices $i$ and $j$ are members of a social network where $X_i$ is the vertex weight corresponding to Person $i$, which indicates whether they want to make a new connection ($X_i = 1$) or not ($X_i = 0$). The vertex weights induce a weight on the edges such that the weight of the edge $i \sim j$ is $(X_i + X_j)$. The quantity $(X_i + X_j)$ can take values in $\{0,1,2\}$. If the value of the sum is $0$, then both $i$ and $j$ are uninterested in making a new connection. If the value is $1$, then one of the parties is interested in making a new connection whereas the other is not. Lastly, if the value is $2$, then both $i$ and $j$ are interested in making a new connection.

In order to apply Theorem \ref{ergmsbm}, we must determine bounds on $D, L_1$, and $L_2$ as defined in the theorem. The gradient of $f$ can be simplified as
	\[ \nabla f(X) = \left(\alpha \left(1 - \frac{1}{n}\right) + \frac{1}{2}\log\left(\frac{p}{1-p}\right) \right) \mathbf{1}_{n}. \]
Since $\nabla f$ is constant, it follows that
	\[ L_1 = \max\left\{ 1, \alpha \left(1 - \frac{1}{n}\right) + \frac{1}{2}\log\left(\frac{p}{1-p}\right) \right\} \]
and $L_2 = 1$. Recall Definition \ref{gc} for the gradient complexity $\mathcal{D}(f)$. For the subgraph counting function in Equation \ref{sums}, we have
	\begin{align*}
		\mathcal{D}(f) &= \ER\left( \sup\left\{ \left(\alpha \left(1 - \frac{1}{n}\right) + \frac{1}{2}\log\left(\frac{p}{1-p}\right)\right) \left\langle \mathbf{1}_n, \Gamma \right\rangle, 0 \right\} \right) \\
			&\le \left|\alpha \left(1 - \frac{1}{n}\right) + \frac{1}{2}\log\left(\frac{p}{1-p}\right)\right| \ER \left| \left\langle \mathbf{1}_n, \Gamma \right\rangle \right|.
	\end{align*}
Note that $Y = \left\langle \mathbf{1}_n, \Gamma \right\rangle$ is a normal random variable with $\mu_{Y} = 0$ and $\sigma_{Y} = \sqrt{n}$. It follows then that $\left| Y \right|$ is a folded normal random variable with $\mu_{\left|Y\right|} = \sqrt{\frac{2n}{\pi}}$. This gives a bound on the gradient complexity as
	\begin{equation}
		\mathcal{D}(f) \le \left|\alpha \left(1 - \frac{1}{n}\right) + \frac{1}{2}\log\left(\frac{p}{1-p}\right)\right| \sqrt{\frac{2n}{\pi}}.
	\end{equation}
Therefore $D = o(n)$ and Theorem \ref{ergmsbm} yields that if $X_{n}^{f}$ is a random vector distributed according to
	\[ \PR\left(X_{n}^{f} = X\right) = \exp\left(f(X)\right)/Z, \]
then $X_{n}^{f}$ is a $(\rho,o(1))$-mixture with $\rho(\mathcal{X}_{f}) = 1 - o(1)$ where
	\[ \mathcal{X}_{f} = \left\{ X \in \overline{\mathcal{C}_{n}}: \norm{X - \frac{\mathbf{1}_{n} + \tanh\left(\nabla f(X) \right)}{2}}_1 \le 5000L_1 L_2^{3/4} D^{1/4} n^{3/4} \right\}. \]

\section*{Acknowledgements}
Mei Yin thanks Patrick Shipman for helpful conversations. The authors would also like to thank Ronen Eldan and Renan Gross for enlightening discussions.

\end{document}